\documentclass{article}

\usepackage[all]{xy}

\usepackage{amssymb}
\newcommand{\NN}{\mathbb{N}}
\newcommand{\QQ}{\mathbb{Q}}
\newcommand{\RR}{\mathbb{R}}
\newcommand{\RCAos}{\mathsf{RCA}_0^{\displaystyle{*}}}
\newcommand{\RCAo}{\mathsf{RCA}_0}
\newcommand{\cl}{\mathrm{cl}}
\newcommand{\ucl}{\mathrm{ucl}}
\newcommand{\WO}{\mathrm{WO}}
\newcommand{\WPO}{\mathrm{WPO}}
\newcommand{\limp}{\Rightarrow}
\newcommand{\liff}{\Leftrightarrow}
\renewcommand{\lnot}{\neg\,}
\newcommand{\nle}{\nleq}
\newcommand{\YD}{\calD}
\newcommand{\calD}{\mathcal{D}}
\newcommand{\calP}{\mathcal{P}}
\newcommand{\calS}{\mathcal{S}}
\newcommand{\calShat}{\widehat{\calS}}
\newcommand{\calSbar}{\overline{\calS}}
\newcommand{\calT}{\mathcal{T}}
\newcommand{\calU}{\mathcal{U}}
\newcommand{\calI}{\mathcal{I}}

\usepackage{amsthm}
\theoremstyle{definition}
\newtheorem{thm}{Theorem}[section]
\newtheorem{dfn}[thm]{Definition}
\newtheorem{lem}[thm]{Lemma}
\newtheorem{rem}[thm]{Remark}

\usepackage[backref=page,colorlinks=true,allcolors=blue]{hyperref}

\begin{document}

\title{Reverse mathematics, Young diagrams,\\
  and the ascending chain condition}

\author{Kostas Hatzikiriakou\\
  Department of Primary Education\\
  University of Thessaly\\
  Argonafton \& Filellinon\\
  Volos 38221, GREECE\\
  \href{mailto:kxatzkyr@uth.gr}{kxatzkyr@uth.gr}\\[10pt]
  Stephen G. Simpson\\
  Department of Mathematics\\
  Pennsylvania State University\\
  State College, PA 16802, USA\\
  \href{http://www.math.psu.edu/simpson}{http://www.math.psu.edu/simpson}\\
  \href{mailto:simpson@math.psu.edu}{simpson@math.psu.edu}}

\date{First draft: January 27, 2014\\
  This draft: \today}

\maketitle

\begin{abstract}
  Let $S$ be the group of finitely supported permutations of a
  countably infinite set.  Let $K[S]$ be the group algebra of $S$ over
  a field $K$ of characteristic $0$.  According to a theorem of
  Formanek and Lawrence, $K[S]$ satisfies the ascending chain
  condition for two-sided ideals.  We study the reverse mathematics of
  this theorem, proving its equivalence over $\RCAo$ (or even over
  $\RCAos$) to the statement that $\omega^\omega$ is well ordered.
  Our equivalence proof proceeds via the statement that the Young
  diagrams form a well partial ordering.
\end{abstract}

\vfill

Simpson's research was partially supported by the Eberly College of
Science at the Pennsylvania State University, and by Simons Foundation
Collaboration Grant 276282.

\newpage

\tableofcontents

\section{Introduction}
\label{sec:intro}

\emph{Reverse mathematics} is a research program in the foundations of
mathematics.  The purpose of reverse mathematics is to discover which
axioms are needed to prove specific core-mathematical theorems of
analysis, algebra, geometry, combinatorics, etc.  Reverse mathematics
takes place in the context of subsystems of second-order arithmetic.
The standard reference for reverse mathematics is \cite[Part
A]{sosoa}.  This paper is a contribution to the reverse mathematics of
algebra.

Let $\QQ$ be the field of rational numbers.  The \emph{Hilbert Basis
  Theorem} \cite{hilbert-hbt} says that, for each positive integer
$n$, there is no infinite ascending sequence of ideals in the
polynomial ring $\QQ[x_1,\ldots,x_n]$.  This is one of the most famous
theorems of abstract algebra.  In a previous paper \cite[Theorem
2.7]{hbt} we have shown that the Hilbert Basis Theorem\footnote{A
  related theorem of Hilbert says that for each $n$ there is no
  infinite ascending sequence of ideals in the power series ring
  $\QQ[[x_1,\ldots,x_n]]$.  We have shown in \cite{hatz-ordinal} that
  this theorem, like the Hilbert Basis Theorem, is
  reverse-mathematically equivalent to $\WO(\omega^\omega)$.} is
reverse-mathematically equivalent to $\WO(\omega^\omega)$.  Here
$\WO(\omega^\omega)$ is the assertion that (the standard system of
Cantor normal form notations for the ordinal numbers less than the
ordinal number) $\omega^\omega$ is well ordered.  The place of
WO$(\omega^\omega)$ within the usual hierarchy of subsystems of
second-order arithmetic \cite{sosoa} is discussed in the expository
paper \cite{WO-vs-IND-arxiv}.

Our reverse-mathematical result in \cite[Theorem 2.7]{hbt} implies
that the Hilbert Basis Theorem is not finitistically reducible.  (See
\cite[Proposition 2.6]{hbt}.)  This foundational outcome is
significant with respect to Hilbert's foundational program of
finitistic reductionism \cite{hilbert-unendliche,partial,tait}.  In a
certain fairly precise philosophical sense
\nocite{iandt}\cite{actual,toim}, our result in \cite[Theorem
2.7]{hbt} justifies Gordan's famous remark (see
\cite{noether-on-gordan}) to the effect that the Hilbert Basis Theorem
is not mathematics but rather theology.

On the other hand, in analyzing the significance of \cite[Theorem
2.7]{hbt}, it is important to note that the Hilbert Basis Theorem
refers not to a single ring but to a sequence of rings
$\QQ[x_1,\ldots,x_n]$ where $n=1,2,3,\ldots$.  Moreover, for each
specific positive integer $n$, the special case $\QQ[x_1,\ldots,x_n]$
of the Hilbert Basis Theorem is finitistically reducible, in fact
provable in $\RCAo$ \cite[Lemmas 3.4 and 3.6]{hbt}.  Therefore, from
the foundational viewpoint of \cite[\S1]{hbt}, it would be interesting
to find a specific commutative ring which has no infinite ascending
sequence of ideals and to prove that the corresponding basis theorem
is not finitistically reducible.

At the moment we are unable to provide such a commutative ring, but in
this paper we achieve something similar in the noncommutative case.
Specifically, we perform a reverse-mathematical analysis of a theorem
of Formanek and Lawrence \cite{formanek-lawrence}.  The
\emph{Formanek/Lawrence Theorem} says that there is no infinite
ascending sequence of two-sided ideals in the group algebra $\QQ[S]$
of the infinite symmetric group $S$ over the field $\QQ$.  The purpose
of this paper is to show that the Formanek/Lawrence Theorem, like the
Hilbert Basis Theorem, is reverse-mathematically equivalent to
$\WO(\omega^\omega)$.  From this result it follows, just as in
\cite{hbt}, that the Formanek/Lawrence Theorem is not finitistically
reducible.

Our work in this paper involves partitions.  A \emph{partition} is a
finite sequence of integers $m_1,\ldots,m_k$ such that
$m_1\ge\cdots\ge m_k>0$ and $k>0$.  It is well known (see for instance
\cite[\S28]{curtis-reiner-1962}) that there is a close relationship
between two-sided ideals in $\QQ[S]$ and partitions.  A partition
$m_1,\ldots,m_k$ is usually visualized as a \emph{Young diagram},
i.e., a planar array of boxes with $k$ left-justified rows of lengths
$m_1,\ldots,m_k$ respectively.  For example, the partition $5,2,2,1$
is visualized as the Young diagram in Figure \ref{fig:5221} consisting
of $10=5+2+2+1$ boxes.  Note that there is another partition
$4,3,1,1,1$ consisting of the lengths of the columns of this same
Young diagram.  A standard reference for partition theory is
\cite{andrews-tofp}.

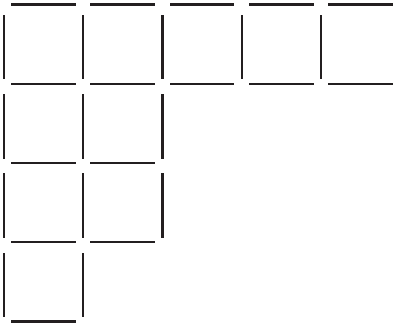
\begin{figure}[hbtf]
  \centering\small
  \[
  \xymatrix{
    \ar@{-}[r]\ar@{-}[d]&\ar@{-}[r]\ar@{-}[d]&\ar@{-}[r]\ar@{-}[d]&  
    \ar@{-}[r]\ar@{-}[d]&\ar@{-}[r]\ar@{-}[d]&\ar@{-}[d]\\
    \ar@{-}[r]\ar@{-}[d]&\ar@{-}[r]\ar@{-}[d]&
    \ar@{-}[r]\ar@{-}[d]&\ar@{-}[r]&\ar@{-}[r]&\\
    \ar@{-}[r]\ar@{-}[d]&\ar@{-}[r]\ar@{-}[d]&\ar@{-}[d]&&\\
    \ar@{-}[r]\ar@{-}[d]&\ar@{-}[r]\ar@{-}[d]&&\\
    \ar@{-}[r]&&&&
  }
  \]
  \caption{The Young diagram corresponding to the partition
    $5,2,2,1$.}
  \label{fig:5221}
\end{figure}

Our notation and terminology in this paper is as follows.  We use
$\NN$ to denote the set of positive integers.  We use $i,j,k,l,m,n$ as
variables ranging over $\NN$.  We use $\YD$ to denote the set of
\emph{diagrams}, i.e., Young diagrams.  We use $D,E,F,G,H$ as
variables ranging over $\YD$.  Given $D\in\YD$, the number of boxes in
$D$ is denoted $|D|$.  We write $r_i(D)=$ the length of the $i$th row
of $D$, or $0$ if $i>$ the number of rows.  Similarly, we write
$c_i(D)=$ the length of the $i$th column of $D$, or $0$ if $i>$ the
number of columns.  Note that
$|D|=r_1(D)+\cdots+r_k(D)=c_1(D)+\cdots+c_l(D)$ where
$k=c_1(D)=\max(\{i\mid r_i(D)>0\})=$ the number of rows, and
$l=r_1(D)=\max(\{i\mid c_i(D)>0\})=$ the number of columns.  Note also
that $|D|\le r_1(D)\cdot c_1(D)$.  For $D,E\in\YD$ we write $D\le E$
to mean that $r_i(D)\le r_i(E)$ for all $i$, or equivalently
$c_i(D)\le c_i(E)$ for all $i$.  Thus $\YD,\le$ is a partial ordering.
We write $D<E$ to mean that $D\le E$ and $D\ne E$.  We write $D\cup
E=\sup(D,E)=$ the unique diagram $F$ such that
$r_i(F)=\max(r_i(D),r_i(E))$ for all $i$, or equivalently
$c_i(F)=\max(c_i(D),c_i(E))$ for all $i$.

As in \cite{hbt} some of our results in this paper involve well
partial ordering theory.  A \emph{well partial ordering} is a partial
ordering $\calP,\le$ such that for all infinite sequences $\langle
P_i\rangle_{i\in\NN}$ of elements of $\calP$ there exist $i,j\in\NN$
such that $i<j$ and $P_i\le P_j$.  For example, any well ordering is a
well partial ordering, and the product of any two well partial
orderings is a well partial ordering.  For background on well partial
ordering theory, see \nocite{mans-weit,log-comb}
\cite{dejongh-parikh,schmidt,bqo-fraisse,baume,vems}.  A key fact for
us in this paper is that $\YD,\le$ is a well partial ordering.  We
denote this fact as $\WPO(\YD)$.  In Theorem \ref{thm:uclacc} we show
that the statement $\WPO(\YD)$ is robust with respect to equivalence
over $\RCAo$.

The plan of this paper is as follows.  In \S\ref{sec:ucl} we prove
that the statements $\WO(\omega^\omega)$ and $\WPO(\YD)$ are
equivalent over $\RCAo$.  In \S\ref{sec:cl} we define what it means
for a subset of $\YD$ to be \emph{closed}, and we prove that
$\WPO(\YD)$ is equivalent over $\RCAo$ to the statement that there is
no infinite ascending sequence of closed subsets of $\YD$.  In
\S\ref{sec:KS} we show that the latter statement is equivalent over
$\RCAo$ to the statement that there is no infinite ascending sequence
of two-sided ideals in $\QQ[S]$.  Combining this with the results of
\S\ref{sec:ucl} and \S\ref{sec:cl}, we obtain our main result.  In
\S\ref{sec:rcaos} we show how the base theory $\RCAo$ can be weakened
to $\RCAos$.  In \S\ref{sec:future} we raise a question for future
research.

\section{Well partial ordering of diagrams}
\label{sec:ucl}

In this section we use a result from \cite{hbt} to prove in $\RCAo$
that the set $\YD$ of all diagrams is well partially ordered if and
only if the ordinal number $\omega^\omega$ is well ordered.

\begin{lem}
  \label{lem:wowpo}
  $\RCAo$ proves $\WO(\omega^\omega)\limp\WPO(\YD)$.
\end{lem}

\begin{proof}
  Reasoning in $\RCAo$, assume that $\WPO(\YD)$ fails, i.e., $\YD$ is
  not well partially ordered.  Let $\langle D_i\rangle_{i\in\NN}$ be a
  \emph{bad sequence} of diagrams, i.e., $\forall i\,\forall
  j\,(i<j\limp D_i\nle D_j)$.  Let $m=r_1(D_1)$ and $n=c_1(D_1)$.  For
  all $j>1$ we have $D_1\nle D_j$, hence $m>r_n(D_j)$ and
  $n>c_m(D_j)$.  To each diagram $D$ such that $m>r_n(D)$ and
  $n>c_m(D)$, we associate a sequence $s(D)\in\NN^{m+n-2}$ given by
  \begin{center}
    $s(D)=\langle
    r_1(D)+1,\ldots,r_{n-1}(D)+1,c_1(D)+1,\ldots,c_{m-1}(D)+1\rangle$.
  \end{center}
  The point is that for any two such diagrams $D$ and $E$, we have
  $D\le E$ if and only if $s(D)\le s(E)$ with respect to the
  coordinatewise partial ordering of $\NN^{m+n-2}$.  Since $\langle
  D_j\rangle_{j>1}$ is a bad sequence of diagrams, it follows that
  $\langle s(D_j)\rangle_{j>1}$ is a bad sequence in $\NN^{m+n-2}$,
  hence $\NN^{m+n-2}$ is not well partially ordered.  It then follows
  by \cite[Lemma 3.6]{hbt} that $\omega^\omega$ is not well ordered,
  i.e., $\WO(\omega^\omega)$ fails.
\end{proof}

\begin{lem}
  \label{lem:wpowo}
  $\RCAo$ proves $\WPO(\YD)\limp\WO(\omega^\omega)$.
\end{lem}

\begin{proof}
  We reason in $\RCAo$.  Given a nonzero ordinal
  $\alpha<\omega^\omega$, write $\alpha$ in Cantor normal form as
  $\alpha=\omega^{n_1}+\cdots+\omega^{n_k}$ where $n_1\ge\cdots\ge
  n_k\ge0$ and $k>0$.  Let $D(\alpha)$ be the unique diagram $D$ with
  $r_1(D)=n_1+1$, \ldots, $r_k(D)=n_{k}+1$ and $r_{k+1}(D)=0$.
  Clearly $D(\alpha)\le D(\beta)$ implies $\alpha\le\beta$.  Assume
  now that $\WO(\omega^\omega)$ fails.  Let
  $\langle\alpha_i\rangle_{i\in\NN}$ be an infinite descending
  sequence of ordinals $<\omega^\omega$, i.e.,
  $\omega^\omega>\alpha_i>\alpha_j$ for all $i,j\in\NN$ such that
  $i<j$.  Then $D(\alpha_i)\nle D(\alpha_j)$ for all $i,j\in\NN$ such
  that $i<j$, hence $\langle D(\alpha_i)\rangle_{i\in\NN}$ is a bad
  sequence in $\YD$, so $\WPO(\YD)$ fails.
\end{proof}

\begin{thm}
  \label{thm:wpowo}
  $\RCAo$ proves $\WPO(\YD)\liff\WO(\omega^\omega)$.
\end{thm}

\begin{proof}
  This is immediate from Lemmas \ref{lem:wowpo} and \ref{lem:wpowo}.
\end{proof}

As a technical supplement to Theorem \ref{thm:wpowo}, we now prove a
theorem to the effect that the statement $\WPO(\YD)$ is robust up to
equivalence over $\RCAo$.

\begin{dfn}
  \label{dfn:ucl}
  Within $\RCAo$ we make the following definitions.
  \begin{enumerate}
  \item A set $\calU$ of diagrams is said to be \emph{upwardly closed}
    if $\forall D\,\forall E\,((D\in\calU$ and $D\le E)\limp
    E\in\calU)$.  For each finite set $\calS$ of diagrams, there
    exists a set $\ucl(\calS)=\{E\mid\exists D\,(D\in\calS$ and
    $D\le E)\}=$ the \emph{upward closure} of $\calS$, i.e., the
    smallest upwardly closed set which includes $\calS$.  An upwardly
    closed set of diagrams is \emph{finitely generated} if it is of
    the form $\ucl(\calS)$ for some finite set $\calS$.
  \item Consider a sequence of diagrams $\langle
    D_i\rangle_{i\in\NN}$.  The sequence is said to be \emph{upwardly
      closed} if $\forall i\,\forall E\,(D_i\le E\limp\exists
    j\,(D_j=E))$.  The sequence is said to be \emph{finitely
      generated} if there is a finite set $\calS$ such that $\forall
    E\,(E\in\ucl(\calS)\liff\exists j\,(D_j=E))$, i.e., there is a
    finitely generated, upwardly closed set which is the range of the
    sequence.
  \end{enumerate}
\end{dfn}

\begin{lem}
  \label{lem:uclseq}
  $\RCAo$ proves the following.  If all upwardly closed sets of
  diagrams are finitely generated, then all upwardly closed sequences
  of diagrams are finitely generated.
\end{lem}

\begin{proof}
  We reason in $\RCAo$.  Let $\langle D_i\rangle_{i\in\NN}$ be an
  upwardly closed sequence of diagrams which is not finitely
  generated.  Define a subsequence $\langle D_{i_n}\rangle_{i\in\NN}$
  recursively as follows.  Let $i_1=1$.  Given $i_1<\cdots<i_n$ let
  $i_{n+1}=$ the least $i>i_n$ such that
  $D_i\notin\ucl(\{D_{i_1},\ldots,D_{i_n}\})$ and
  $|D_i|>\max(|D_{i_1}|,\ldots,|D_{i_n}|)$.  Thus $i_n<i_{n+1}$ and
  $|D_{i_n}|<|D_{i_{n+1}}|$ for all $n$.  By $\Delta^0_1$
  comprehension let $\calU$ be the set of diagrams $D$ such that
  $D\in\ucl(\{D_{i_1},\ldots,D_{i_n}\})$ where $n=$ the least $n$ such
  that $|D_{i_n}|>|D|$.  Clearly
  $\calU=\bigcup_{n\in\NN}\ucl(\{D_{i_1}\ldots,D_{i_n}\})$, hence
  $\calU$ is upwardly closed, and
  $\calU\ne\ucl(\{D_{i_1},\ldots,D_{i_n}\})$ for all $n$, hence
  $\calU$ is not finitely generated.
\end{proof}

\begin{thm}
  \label{thm:uclacc}
  Over $\RCAo$ each of the following statements implies the others.
  \begin{enumerate}
  \item $\WPO(\YD)$.
  \item There is no infinite ascending sequence of upwardly closed
    sets of diagrams.
  \item There is no infinite ascending sequence of finitely generated,
    upwardly closed sets of diagrams.
  \item Every upwardly closed set of diagrams is finitely generated.
  \item Every upwardly closed sequence of diagrams is finitely
    generated.
  \end{enumerate}
\end{thm}

\begin{proof}
  We reason in $\RCAo$.  To prove $1\limp2$, assume that $2$ fails.
  Let $\langle\calU_i\rangle_{i\in\NN}$ be an infinite ascending
  sequence of upwardly closed sets of diagrams.  For each $i$ let
  $D_i$ be the least element of $\calU_{i+1}\setminus\calU_i$.  Then
  $\langle D_i\rangle_{i\in\NN}$ is a bad sequence of diagrams, so $1$
  fails.

  Trivially $2\limp3$.  To prove $3\limp4$, assume that $4$ fails.
  Let $\calU$ be an upwardly closed set of diagrams which is not
  finitely generated.  Let $D_1$ be the least element of $\calU$.
  Given $D_1,\ldots,D_n\in\calU$, let $D_{n+1}=$ the least element of
  $\calU\setminus\calU_n$ where $\calU_n=\ucl(\{D_1,\ldots,D_n\})$.
  Then $\langle\calU_n\rangle_{n\in\NN}$ is an infinite ascending
  sequence of upwardly closed sets, so $3$ fails.

  By Lemma \ref{lem:uclseq} we have $4\limp5$.  It remains to prove
  $5\limp1$.  Assume that $1$ fails.  Let $\langle
  D_i\rangle_{i\in\NN}$ be a bad sequence of diagrams.  The formula
  $\Phi(E)\equiv\exists i\,(D_i\le E)$ is $\Sigma^0_1$, so by
  \cite[Lemma II.3.7]{sosoa} there is a sequence of diagrams $\langle
  E_j\rangle_{j\in\NN}$ such that $\forall E\,(\Phi(E)\liff\exists
  j\,(E=E_j))$.  Clearly $\langle E_j\rangle_{j\in\NN}$ is an upwardly
  closed sequence of diagrams, and by $\Sigma^0_1$ bounding we have
  $\forall m\,\exists n\,(\forall j<m)\,(\exists i<n)\,(D_i\le
  E_j)$.  Because $\langle D_i\rangle_{i\in\NN}$ is bad, it follows
  that $\langle E_j\rangle_{j\in\NN}$ is not finitely generated, so
  $5$ fails.  This completes the proof.
\end{proof}

\section{Closed sets of diagrams}
\label{sec:cl}

\begin{dfn}
  A set $\calU$ of diagrams is said to be \emph{closed} if
  \begin{center}
    $\forall D\,(D\in\calU\liff\forall E\,(E>D\limp E\in\calU))$.
  \end{center}
\end{dfn}

\begin{rem}
  Obviously all closed sets of diagrams are upwardly closed.  However,
  not all upwardly closed sets of diagrams are closed.  For example,
  letting $D$ and $E$ be the diagrams corresponding to the partitions
  $4,2$ and $2,2,1,1$ respectively, the upwardly closed set
  $\ucl(\{D,E\})$ is not closed.  In this section we clarify the
  structure of closed sets of diagrams.  Definition \ref{dfn:DrDc} and
  Lemmas \ref{lem:cl0}--\ref{lem:cl2} and Theorem \ref{thm:cl} appear
  to be new, in the sense that we have not found their counterparts in
  the partition theory literature such as \cite{andrews-tofp,
    curtis-reiner-1962,formanek-lawrence,fulton-tableaux-1997}.
\end{rem}

\begin{dfn}
  \label{dfn:DrDc}
  Given a diagram $D$, let $(D)_r=$ the diagram obtained by shortening
  the first row to the length of the second row, or $1$ if there is no
  second row.  Thus $r_1((D)_r)=\max(r_2(D),1)$ and
  $r_i((D)_r)=r_i(D)$ for all $i>1$.  Similarly, let $(D)_c=$ the
  diagram obtained by shortening the first column to the length of the
  second column, or $1$ if there is no second column.  Thus
  $c_1((D)_c)=\max(c_2(D),1)$ and $c_i((D)_c)=c_i(D)$ for all $i>1$.
\end{dfn}

The next three lemmas are proved in $\RCAo$.

\begin{lem}
  \label{lem:cl0}
  For all diagrams $D$ we have $D=(D)_r\cup(D)_c$.  For all diagrams
  $D$ and $E$, if $(D)_r=(E)_r$ or $(D)_c=(E)_c$, then $D\le E$
  or $E\le D$.
\end{lem}

\begin{proof}
  These statements are obvious from Definition \ref{dfn:DrDc}.
\end{proof}

\begin{lem}
  \label{lem:cl1}
  Let $D,E,F$ be diagrams such that $|F|>(r_1(D)-1)(c_1(E)-1)$ and
  $F\ge(D)_r\cup(E)_c$.  Then $F\ge D$ or $F\ge E$.
\end{lem}

\begin{proof}
  Since $|F|>(r_1(D)-1)(c_1(E)-1)$, we must have $r_1(F)\ge r_1(D)$ or
  $c_1(F)\ge c_1(E)$.  In the first case we have $F\ge($first row of
  $D)$ and $F\ge(D)_r$, hence $F\ge D$.  In the second case we have
  $F\ge($first column of $E)$ and $F\ge(E)_c$, hence $F\ge E$.
\end{proof}

\begin{lem}
  \label{lem:cl2}
  Let $\calS$ be a finite set of diagrams.  Suppose $F$ is a diagram
  such that $\exists n\,\forall G\,((|G|>n$ and $G>F)\limp
  G\in\ucl(\calS))$.  Then $F\ge(D)_r\cup(E)_c$ for some
  $D,E\in\calS$.
\end{lem}

\begin{proof}
  Let $\calS$ and $F$ and $n$ be as in the hypothesis.  Since $\calS$
  is finite, we may safely assume that
  $n>\max(|F|,\max(\{r_1(D),c_1(E)\mid D,E\in\calS\}))$.  Let $G$ be
  the unique diagram with $r_1(G)=n+1$ and $r_i(G)=r_i(F)$ for all
  $i>1$.  Then $G>F$ and $|G|>n$, hence $G\in\ucl(\calS)$, i.e., $G\ge
  D$ for some $D\in\calS$.  It then follows that $F\ge(D)_r$.
  Similarly, let $H$ be the unique diagram with $c_1(H)=n+1$ and
  $c_i(H)=c_i(F)$ for all $i>1$, then $H\ge E$ for some $E\in\calS$,
  hence $F\ge(E)_c$.  Thus $F\ge(D)_r\cup(E)_c$.
\end{proof}

\begin{dfn}
  \label{dfn:Sbar}
  Given a finite set $\calS$ of diagrams, let
  \begin{center}
    $\calShat=\{(D)_r\cup(E)_c\mid D,E\in\calS\}$,
  \end{center}
  and let
  \begin{center}
    $\calSbar=\{F\in\ucl(\calShat)\mid|F|\le\|\calS\|\}$
  \end{center}
  where
  \begin{center}
    $\|\calS\|=\max(\{(r_1(D)-1)(c_1(E)-1)\mid D,E\in\calS\})$.
  \end{center}
  Note that $\calShat$ and $\calSbar$ are finite sets of diagrams, and
  they can be found effectively given the finite set $\calS$.  Note
  also that $\calS\subseteq\calShat$, hence
  $\ucl(\calS)\subseteq\ucl(\calShat)$.
\end{dfn}

\begin{thm}
  \label{thm:cl}
  $\RCAo$ proves the following.  For any finite set $\calS$ of
  diagrams, we have $\ucl(\calShat)=\calSbar\cup\ucl(\calS)$.
  Moreover, $\ucl(\calShat)$ is closed and includes $\calS$ and is
  included in all closed sets that include $\calS$.
\end{thm}

\begin{proof}
  Trivially $\calSbar\cup\ucl(\calS)\subseteq\ucl(\calShat)$, and by
  Lemma \ref{lem:cl1} we have
  $\ucl(\calShat)\subseteq\calSbar\cup\ucl(\calS)$, so
  $\ucl(\calShat)=\calSbar\cup\ucl(\calS)$.  Also by Lemma
  \ref{lem:cl1}, for any closed set $\calU$ such that
  $\calS\subseteq\calU$ we have $\calShat\subseteq\calU$, hence
  $\ucl(\calShat)\subseteq\calU$.  It remains to prove that
  $\ucl(\calShat)$ is closed.  Let $F$ be a diagram such that $\forall
  G\,(G>F\limp G\in\ucl(\calShat))$, i.e., $\forall G\,(G>F\limp
  G\in\calSbar\cup\ucl(\calS))$.  From the definition of $\calSbar$ it
  follows that $\forall G\,((|G|>\|\calS\|$ and $G>F)\limp
  G\in\ucl(\calS))$.  But then by Lemma \ref{lem:cl2} we have
  $F\in\ucl(\calShat)$.  Thus $\ucl(\calShat)$ is closed, Q.E.D.
\end{proof}

We shall now use Theorem \ref{thm:cl} to prove a result about closed
sets of diagrams which is analogous to Theorem \ref{thm:uclacc} about
upwardly closed sets of diagrams.  Our result here will be needed in
\S\ref{sec:KS}.

\begin{dfn}
  \label{dfn:cl}
  Within $\RCAo$ we make the following definitions.
  \begin{enumerate}
  \item The \emph{closure} of a finite set $\calS$ of diagrams is the
    set
    \begin{center}
      $\cl(\calS)=\ucl(\calShat)=\calSbar\cup\ucl(\calS)$.
    \end{center}
    This definition is appropriate in view of Definition
    \ref{dfn:Sbar} and Theorem \ref{thm:cl}.
  \item A sequence of diagrams $\langle D_i\rangle_{i\in\NN}$ is said
    to be \emph{closed} if $\forall D\,((\exists
    i\,(D=D_i))\liff\forall E\,(E>D\limp\exists j\,(E=D_j)))$.
  \end{enumerate}
\end{dfn}

\begin{rem}
  \label{rem:cl}
  Theorem \ref{thm:cl} implies that a closed set of diagrams is
  finitely generated qua closed set if and only if it is finitely
  generated qua upwardly closed set, i.e., finitely generated in the
  sense of Definition \ref{dfn:ucl}.  And similarly, a closed sequence
  of diagrams is finitely generated qua closed sequence if and only if
  it is finitely generated qua upwardly closed sequence.
\end{rem}

\begin{lem}
  \label{lem:cl3}
  $\RCAo$ proves the following.  Given a closed sequence of diagrams
  $\langle D_i\rangle_{i\in\NN}$ which is not finitely generated, we
  can find a subsequence $\langle D_{i_n}\rangle_{n\in\NN}$ such that
  $\forall n\,(D_{i_n}\notin\cl(\{D_{i_1},\ldots,D_{i_{n-1}}\}))$ and
  $\forall m\,\forall n\,(m<n\limp D_{i_n}\nle D_{i_m})$.
\end{lem}

\begin{proof}
  We reason in $\RCAo$.  Using the results of \cite[\S II.3]{sosoa},
  define $i_n$ recursively as follows.  Let $i_1=1$.  Assume
  inductively that $i_1<\cdots<i_n$ have been defined.  By bounded
  $\Sigma^0_1$ comprehension \cite[Theorem II.3.9]{sosoa} there is a
  finite set $\calT_n=\{D\mid\exists j\,\exists m\,(D=D_j$ and $m\le
  n$ and $D\le D_{i_m})\}$.  Since $\langle D_i\rangle_{i\in\NN}$ is
  not finitely generated, we have $D_i\notin\cl(\calT_n)$ for
  infinitely many $i$.  Let $i_{n+1}=$ the least such $i$ which is
  also $>i_n$.  Clearly the subsequence $\langle
  D_{i_n}\rangle_{n\in\NN}$ has the desired properties.
\end{proof}

\begin{lem}
  \label{lem:clseq}
  $\RCAo$ proves the following.  If all closed sets of diagrams are
  finitely generated, then all closed sequences of diagrams are
  finitely generated.
\end{lem}

\begin{proof}
  We reason in $\RCAo$.  Suppose there is a closed sequence of
  diagrams which is not finitely generated.  By Lemma \ref{lem:cl3}
  let $\langle D_n\rangle_{n\in\NN}$ be a sequence of diagrams such
  that $\forall n\,(D_n\notin\cl(\{D_1,\ldots,D_{n-1}\}))$ and
  $\forall m\,\forall n\,(m<n\limp D_n\nle D_m)$.  In particular we
  have $\forall m\,\forall n\,(D_m\le D_n\limp m=n)$.  It follows by
  Lemma \ref{lem:cl0} that for all $D$, there is at most one $n$ such
  that $(D_n)_r=D$ and there is at most one $n$ such that $(D_n)_c=D$.

  Using the results of \cite[\S II.3]{sosoa} define a subsequence
  $\langle D_{n_k}\rangle_{k\in\NN}$ recursively as follows.  Assume
  inductively that $n_1<\cdots<n_{k-1}$ have been defined.  Let
  $\calS_k=\{D_{n_1},\ldots,D_{n_{k-1}}\}$.  The set of diagrams $D$
  such that $|D|\le\|\calS_k\|$ is finite, so by bounded $\Sigma^0_1$
  comprehension and $\Sigma^0_1$ bounding we have
  $|(D_n)_r|>\|\calS_k\|$ and $|(D_n)_c|>\|\calS_k\|$ for all
  sufficiently large $n$.  Let $n_k=$ the least such $n$ which is also
  $>n_{k-1}$.  This completes the construction.  The construction
  insures that $\forall k\,\forall l\,\forall D\,($if $k<l$ and
  $D\in\cl(\calS_l)$ and $|D|\le\|\calS_k\|$ then $D\in\cl(\calS_k))$.

  By $\Delta^0_1$ comprehension let $\calU$ be the set of diagrams $D$
  such that $D\in\cl(\calS_k)$ where $k=$ the least $k$ such that
  $|D|\le\|\calS_k\|$.  Then $\calU=\bigcup_{k\in\NN}\cl(\calS_k)$ is
  a closed set of diagrams which not finitely generated.  This proves
  our lemma.
\end{proof}

\begin{thm}
  \label{thm:clacc}
  Over $\RCAo$ each of the following statements implies the others.
  \begin{enumerate}
  \item $\WPO(\YD)$.
  \item There is no infinite ascending sequence of closed
    sets of diagrams.
  \item There is no infinite ascending sequence of finitely generated
    closed sets of diagrams.
  \item Every closed set of diagrams is finitely generated.
  \item Every closed sequence of diagrams is finitely generated.
  \end{enumerate}
\end{thm}

\begin{proof}
  We reason in $\RCAo$.  Let $D$ be a diagram, and let $\calS$ be a
  finite set of diagrams.  By Theorem \ref{thm:cl} and Definition
  \ref{dfn:cl}, the relation $D\in\cl(\calS)$ is $\Delta^0_1$.  With
  this remark in mind, the proofs of $1\limp2$ and $2\limp3$ and
  $3\limp4$ are similar to the proofs of the corresponding
  implications in Theorem \ref{thm:uclacc}.

  By Lemma \ref{lem:clseq} we have $4\limp5$.  It remains to prove
  $5\limp1$.  Assume that $1$ fails.  Let $\langle
  D_i\rangle_{i\in\NN}$ be a bad sequence of diagrams.  The formula
  $\Phi(E)\equiv\exists i\,(E\in\cl(\{D_1,\ldots,D_i\}))$ is
  $\Sigma^0_1$, so by \cite[Lemma II.3.7]{sosoa} there is a sequence
  of diagrams $\langle E_j\rangle_{j\in\NN}$ such that $\forall
  E\,(\Phi(E)\liff\exists j\,(E=E_j))$.  Clearly $\langle
  E_j\rangle_{j\in\NN}$ is closed.  We claim that $\langle
  E_j\rangle_{j\in\NN}$ is not finitely generated.  Otherwise, by
  $\Sigma^0_1$ bounding there would be an $n$ such that
  $\{D_1,\ldots,D_n\}$ is a set of generators.  Let
  $\calS=\{D_1,\ldots,D_n\}$.  By Theorem \ref{thm:cl} and Definition
  \ref{dfn:cl} we have $\cl(\calS)=\calSbar\cup\ucl(\calS)$.  Since
  $\calSbar$ is finite we have $D_i\notin\calSbar$ for all
  sufficiently large $i$, and since $\langle D_i\rangle_{i\in\NN}$ is
  bad we have $D_i\notin\ucl(\calS)$ for all $i>n$.  It is now clear
  that $D_i\notin\cl(\calS)$ for all sufficiently large $i$,
  contradicting our assumption that $\calS$ is a set of generators.
  Thus $\langle E_j\rangle_{j\in\NN}$ is a counterexample to $5$, so
  $5$ fails, Q.E.D.
\end{proof}

\begin{rem}
  \label{rem:Dn}
  For each $n\in\NN$ let $\calD_n$ be the set of nonempty finite
  subsets of $\NN^n$ which are downwardly closed with respect to the
  coordinatewise partial ordering of $\NN^n$.  It is known that
  $\calD_n$ is well partially ordered under inclusion.  Recall that
  Figure \ref{fig:5221} is the diagram corresponding to the partition
  $5,2,2,1$.  If we rotate Figure \ref{fig:5221} counterclockwise by
  $3\pi/4$ radians, we obtain a picture of an element of $\calD_2$,
  namely, the downwardly closed subset of $\NN^2$ consisting of 10
  elements:
  \begin{center}
    $(1,1),\,(1,2),\,(1,3),\,(1,4),\,(1,5),
    \,(2,1),\,(2,2),\,(3,1),\,(3,2),\,(4,1)$.
  \end{center}
  In this way we obtain a one-to-one order-preserving correspondence
  between $\calD_2$ and $\YD$.  Similarly, there is a one-to-one
  order-preserving correspondence between $\calD_{n+1}$ and the
  so-called $n$-dimensional partitions which are discussed in
  \cite[Chapter 11]{andrews-tofp}.  Our results in this section may be
  interpreted as results about $\calD_2$, and it is straightforward to
  generalize them to $\calD_{n+1}$.  It would be interesting to find
  appropriate similar generalizations of our results in \S2, from
  $\calD_2$ to $\calD_{n+1}$.
\end{rem}

\section{Two-sided ideals in $K[S]$}
\label{sec:KS}

In this section we review the known one-to-one correspondence
\cite[Theorem 2]{formanek-lawrence} between closed sets of diagrams
and two-sided ideals of $K[S]$.  We observe that this correspondence
is provable in $\RCAo$.  We use this observation plus our Theorems
\ref{thm:wpowo} and \ref{thm:clacc} to obtain a reverse-mathematical
classification of the known theorem \cite[Theorem
10]{formanek-lawrence} that $K[S]$ satisfies the ascending chain
condition for two-sided ideals.

\begin{dfn}
  \label{dfn:KSKSn}
  The following definitions are made in $\RCAo$.
  \begin{enumerate}
  \item Let $S$ be the \emph{infinite symmetric group}, i.e., the
    group of finitely supported permutations of $\NN$.  Let $K$ be a
    countable field of characteristic $0$.  For instance, we could
    take $K=\QQ$.  The \emph{group algebra} $K[S]$ is the set of
    formal finite linear combinations $\sum_ga_gg$ where $g\in S$,
    $a_g\in K$, and $a_g=0$ for all but finitely many $g\in S$.  The
    algebraic operations on $K[S]$ are given by
    \begin{center}
      $c\sum_ga_gg=\sum_g(ca_g)g$,\\[10pt]
      $\sum_ga_gg+\sum_gb_gg=\sum_g(a_g+b_g)g$,\\[10pt]
      $\sum_ga_gg\cdot\sum_gb_gg=\sum_g\sum_h(a_gb_h)(gh)$.\\
      {\ }
    \end{center}
  \item A \emph{two-sided ideal of} $K[S]$ is a nonempty set
    $\calI\subseteq K[S]$ such that $(\forall p\in\calI)\,(\forall
    q\in\calI)\,(p+q\in\calI)$ and $(\forall p\in\calI)\,(\forall r\in
    K[S])\,(p\cdot r\in\calI$ and $r\cdot p\in\calI)$.
  \item For each $n\in\NN$ let $S_n$ be the group of permutations of
    $\{1,\ldots,n\}$.  The group algebra $K[S_n]$ and the two-sided
    ideals of $K[S_n]$ are defined as above, replacing $S$ by $S_n$.
    Identifying permutations of $\{1,\ldots,n\}$ with permutations of
    $\NN$ having support included in $\{1,\ldots,n\}$, we have
    $S=\bigcup_{n\in\NN}S_n$ and $K[S]=\bigcup_{n\in\NN}K[S_n]$.
  \end{enumerate}
\end{dfn}

\noindent
The next lemma summarizes some known facts from the representation
theory of $S_n$.

\begin{lem}
  \label{lem:KSn}
  $\RCAo$ proves the following.  There is an explicit one-to-one
  mapping $e:\YD\to K[S]$ with the following properties.
  \begin{enumerate}
  \item\label{it:eD} The irreducible central idempotents of $K[S_n]$
    are $e(D)$, $|D|=n$.
  \item\label{it:eDideals} Consequently, there is a one-to-one
    correspondence between two-sided ideals of $K[S_n]$ and sets of
    diagrams of size $n$.  The two-sided ideal of $K[S_n]$
    corresponding to $\calS\subseteq\{D\mid|D|=n\}$ is generated by
    $\{e(D)\mid D\in\calS\}$.
  \item\label{it:branching} Let $E$ be a diagram of size $\le n$.
    Then $e(E)$ and $\{e(D)\mid|D|=n,E\le D\}$ generate the same
    two-sided ideal of $K[S_n]$.
  \end{enumerate}
\end{lem}

\begin{proof}
  As noted in \cite[Theorem 1]{formanek-lawrence}, these facts are
  proved in \cite[4.27, 4.51, 4.52]{kerber-1}.  Their formalization
  within $\RCAo$ is routine.  Some other helpful sources are
  \cite[\S\S23--28]{curtis-reiner-1962} for facts \ref{it:eD} and
  \ref{it:eDideals}, and \cite[Chapter 7]{fulton-tableaux-1997} for
  fact \ref{it:branching}.
\end{proof}

\begin{lem}
  \label{lem:KS}
  $\RCAo$ proves the following.  There is a one-to-one correspondence
  between two-sided ideals of $K[S]$ and closed sets of diagrams.  The
  two-sided ideal of $K[S]$ corresponding to a closed set $\calU$ of
  diagrams is generated by $\{e(D)\mid D\in\calU\}$.  A two-sided
  ideal of $K[S]$ is finitely generated if and only if the
  corresponding closed set of diagrams is finitely generated.
\end{lem}

\begin{proof}
  We reason in $\RCAo$.  Define $\calT\subseteq\{E\mid|E|\le n\}$ to
  be \emph{$n$-closed} if 
  \begin{center}
    $\calT=\{E\mid|E|\le n,\forall D\,((|D|=n,E\le D)\limp
    D\in\calS)\}$
  \end{center}
  for some $\calS\subseteq\{D\mid|D|=n\}$.  Lemma \ref{lem:KSn}
  implies a one-to-one correspondence between $n$-closed subsets of
  $\{E\mid|E|\le n\}$ and two-sided ideals $\calI$ of $K[S_n]$, where
  the $n$-closed set corresponding to $\calI$ is $\{E\mid|E|\le
  n,e(E)\in\calI\}$.  Our lemma then follows upon noting that (1) a
  set $\calU$ of diagrams is closed if and only if $\forall
  n\,(\calU\cap\{E\mid|E|\le n\}$ is $n$-closed$)$, and (2) a set
  $\calI\subseteq K[S]$ is a two-sided ideal of $K[S]$ if and only if
  $\forall n\,(\calI\cap K[S_n]$ is a two-sided ideal of $K[S_n])$.
\end{proof}

We now present the main theorem of this paper.

\begin{thm}
  \label{thm:KSacc}
  Over $\RCAo$ each of the following statements implies the others.
  \begin{enumerate}
  \item $\WO(\omega^\omega)$.
  \item There is no infinite ascending sequence of two-sided ideals of
    $K[S]$.
  \item There is no infinite ascending sequence of finitely generated,
    two-sided ideals of $K[S]$.
  \item Every two-sided ideal of $K[S]$ is finitely generated.
  \end{enumerate}
\end{thm}

\begin{proof}
  We reason in $\RCAo$.  The one-to-one correspondence in Lemma
  \ref{lem:KS} is such that each two-sided ideal of $K[S]$ is
  uniformly $\Delta^0_1$ definable from its corresponding closed set
  of diagrams and vice versa.  Therefore, by $\Delta^0_1$
  comprehension, the one-to-one correspondence between two-sided
  ideals $\calI\subseteq K[S]$ and closed sets $\calU\subseteq\YD$
  extends to a one-to-one correspondence between sequences
  $\langle\calI_i\rangle_{i\in\NN}$ of two-sided ideals of $K[S]$ and
  sequences $\langle\calU_i\rangle_{i\in\NN}$ of closed sets of
  diagrams.  We now see that Theorem \ref{thm:KSacc} follows from
  Theorems \ref{thm:wpowo} and \ref{thm:clacc}.
\end{proof}

\section{Weakening the base theory}
\label{sec:rcaos}

The usual base theory for reverse mathematics \cite[Part A]{sosoa} is
$\RCAo$.  However, there is an alternative base theory $\RCAos$
consisting of $\RCAo$ with $\Sigma^0_1$ induction replaced by
exponentiation plus $\Sigma^0_0$ induction \cite[\S IX.4]{sosoa}.
Many of the known reversals over $\RCAo$ can be proved over the weaker
base theory $\RCAos$, and in this way one strengthens the reversals
\cite[Remark X.4.3]{sosoa}.  There are also some reversals over
$\RCAos$ which have no counterpart over $\RCAo$, because the theorems
in question are provable in $\RCAo$ and in fact equivalent to $\RCAo$
over $\RCAos$.  The alternative base theory $\RCAos$ was introduced
and used in \cite{simpson-smith} and has been used in several
subsequent publications including
\cite{hatz-disguises,hatz-phd,ko-yo,bct,rmpc,yokoyama-rt-weak}.

The purpose of this section is to strengthen some of our reversals in
Theorems \ref{thm:uclacc} and \ref{thm:clacc} and \ref{thm:KSacc} by
weakening the base theory from $\RCAo$ to $\RCAos$.  In order to do
so, we exercise care in defining the concept ``infinite ascending
sequence'' within $\RCAos$.  The following definition of ``infinite
ascending sequence'' is equivalent over $\RCAo$ to the usual
definition with $I=\NN$, but the equivalence does not hold over
$\RCAos$.

\begin{dfn}
  \label{dfn:infseq}
  Within $\RCAos$ we define (a code for) an \emph{infinite sequence of
    ideals of $K[S]$} to consist of an infinite set $I\subseteq\NN$
  together with a set $\calI\subseteq K[S]\times I$ such that for each
  $i\in I$ the set $\calI_i=\{p\mid(p,i)\in\calI\}$ is an ideal of
  $K[S]$.  Such a sequence is denoted $\langle\calI_i\rangle_{i\in
    I}$.  The sequence is said to be \emph{ascending} if $(\forall
  i\in I)\,(\forall j\in I)\,(i<j\limp\calI_i\subsetneqq\calI_j)$.
  Similarly within $\RCAos$ we define (codes for) infinite (ascending)
  sequences $\langle\calU_i\rangle_{i\in I}$ of (closed or upwardly
  closed) sets of diagrams.
\end{dfn}

\begin{lem}
  \label{lem:descseq}
  Over $\RCAos$ each of the following statements implies the others.
  \begin{enumerate}
  \item $\RCAo$.
  \item $\Sigma^0_1$ induction.
  \item There is no infinite descending sequence in $\NN$.  In other
    words, there is no function $f:I\to\NN$ such that $I\subseteq\NN$
    is infinite and $(\forall i\in I)\,(\forall j\in I)\,(i<j\limp
    f(i)>f(j))$.
  \end{enumerate}
\end{lem}

\begin{proof}
  We reason in $\RCAos$.  For $1\liff2$ see \cite{simpson-smith} or
  \cite[\S X.4]{sosoa}.  Clearly we have $2\limp3$, so it remains to
  prove $3\limp2$.  Suppose $2$ fails.  Let $\Phi(m)$ be a
  $\Sigma^0_1$ formula such that $\Phi(1)$ and $\forall
  m\,(\Phi(m)\limp\Phi(m+1))$ and $\exists n\,\lnot\Phi(n)$.  Write
  $\Phi(m)\equiv\exists j\,\Theta(m,j)$ where $\Theta(m,j)$ is
  $\Sigma^0_0$.  By $\Sigma^0_0$ comprehension there is a function
  $g:\NN\to\{1,\ldots,n\}$ defined by $g(i)=$ the least $m$ such that
  $\lnot(\exists j<i)\,\Theta(m,j)$.  Clearly we have $\forall
  i\,\forall j\,(i>j\limp g(i)\ge g(j))$.  By $\Delta^0_1$
  comprehension let $I=\{i+1\mid g(i+1)>g(i)\}$.  If the set $I$ were
  finite, then letting $m_0+1=g(i_0+1)$ where $i_0+1=$ the largest
  element of $I$, we would have $\Phi(m_0)$ and $\lnot\Phi(m_0+1)$, a
  contradiction.  Thus $I$ is infinite, so $f:I\to\NN$ defined by
  $f(i)=n-g(i)$ is an infinite descending sequence, so $3$ fails,
  Q.E.D.
\end{proof}

\begin{lem}
  \label{lem:RCAos}
  Lemmas \ref{lem:KSn} and \ref{lem:KS} and Theorem \ref{thm:cl} hold
  over $\RCAos$.
\end{lem}

\begin{proof}
  The proofs in $\RCAo$ remain valid in $\RCAos$.  The proof of Lemma
  \ref{lem:KSn} involves Gaussian elimination for finite systems of
  linear equations over $K$, but this works perfectly well in $\RCAos$
  despite the use of $\Sigma^0_1$ induction in \cite[Exercise
  II.4.11]{sosoa}.  The difference here is that $K$, unlike the real
  field $\RR$, is a countable discrete field.
\end{proof}

The main result of this section is as follows.
\begin{thm}
   Over $\RCAos$ each of the following statements implies the others.
  \begin{enumerate}
  \item $\RCAo+\WO(\omega^\omega)$.
  \item There is no infinite ascending sequence of upwardly closed sets of
    diagrams.
  \item There is no infinite ascending sequence of closed sets of
    diagrams.
  \item There is no infinite ascending sequence of two-sided ideals of
    $K[S]$.
  \end{enumerate}
\end{thm}

\begin{proof}
  We reason in $\RCAos$.  By Theorems \ref{thm:uclacc} and
  \ref{thm:clacc} and \ref{thm:KSacc}, it will suffice to prove that
  the disjunction $2\lor3\lor4$ implies $\RCAo$.  Suppose $\RCAo$
  fails.  By Lemma \ref{lem:descseq} there is an infinite descending
  sequence $f:I\to\NN$.  For each $i\in I$ let $F_i$ be the diagram
  such that $r_1(F_i)=f(i)$ and $r_2(F_i)=0$.  By Lemma
  \ref{lem:RCAos} let $\calI_i$ be the two-sided ideal of $K[S]$
  corresponding to the closed set
  $\calU_i=\cl(\{F_i\})=\ucl(\{F_i\})$.  Then
  $\langle\calU_i\rangle_{i\in I}$ is an infinite ascending sequence
  of (upwardly) closed sets of diagrams, and
  $\langle\calI_i\rangle_{i\in I}$ is an infinite ascending sequence
  of two-sided ideals in $K[S]$.  Thus $2\lor3\lor4$ fails, Q.E.D.
\end{proof}

In a similar fashion, we can strengthen \cite[Theorem 2.7]{hbt} by
weakening the base theory from $\RCAo$ to $\RCAos$.  As in \cite{hbt}
let $K[x_1,\ldots,x_n]$ be the ring of polynomials over $K$ in $n$
indeterminates.

\begin{thm}
    Over $\RCAos$ each of the following statements implies the other.
    \begin{enumerate}
    \item $\RCAo+\WO(\omega^\omega)$.
    \item For each $n\in\NN$ there is no infinite ascending sequence
      of ideals in $K[x_1,\ldots,x_n]$.
    \end{enumerate}
\end{thm}

\begin{proof}
  We reason in $\RCAos$.  By \cite[Theorem 2.7]{hbt} we have
  $1\limp2$.  Assume that $2$ holds and $1$ fails.  Then by
  \cite[Theorem 2.7]{hbt} $\RCAo$ must fail, so by Lemma
  \ref{lem:descseq} let $f:I\to\NN$ be an infinite descending
  sequence.  Let $K[x]$ be the ring of polynomials over $K$ in one
  indeterminate.  For each $i\in I$ let $\calI_i$ be the ideal in
  $K[x]$ generated by the monomial $x^{m_i}$ where $m_i=2^{f(i)}$.
  Then $\langle\calI_i\rangle_{i\in I}$ is an infinite ascending
  sequence of ideals in $K[x]$, so $2$ fails for $n=1$.  This
  contradiction completes the proof.
\end{proof}

\section{A question for future research}
\label{sec:future}

In this section we raise a question for future research.

Maclagen\footnote{We thank Florian Pelupessy for calling our attention
  to
  \cite{monomial-ideals-aschenbrenner-pong,monomial-ideals-maclagen}.}
\cite{monomial-ideals-maclagen} has shown that monomial ideals in the
polynomial ring $K[x_1,\ldots,x_n]$ satisfy the \emph{antichain
  condition}, i.e., there is no infinite sequence of such ideals which
are pairwise incomparable under inclusion.  Maclagen's argument also
gives the same result for the power series ring $K[[x_1,\ldots,x_n]]$.
And in a similar vein we have the following result, which is
apparently new.

\begin{thm}
  \label{thm:KSwpo}
  Let $K[S]$ be as in \S\ref{sec:KS}.  Then, the two-sided ideals in
  $K[S]$ satisfy the antichain condition.  In other words, there is no
  infinite sequence of two-sided ideals in $K[S]$ which are pairwise
  incomparable under inclusion.
\end{thm}

\begin{proof}
  The \emph{better partial orderings} \cite{nw-seq} are a subclass of
  the well partial orderings.  Many of the well partial orderings
  which arise in practice are actually better partial orderings.
  However, the class of better partial orderings enjoys many
  infinitary closure properties which are not shared by the class of
  well partial orderings.  A key result of this kind due to
  Nash-Williams\footnote{See also our exposition in
    \cite{bqo-fraisse}.} \cite{nw-seq} implies the following:
  \begin{equation}                                                              
    \label{eq:dcl}                                                          
    \parbox{4.2in}{If $\calP$ is a better partial ordering, then the
      downwardly closed subsets of $\calP$ form a better partial
      ordering under inclusion.}
  \end{equation}
  Taking complements, we also have:
  \begin{equation}                                                              
    \label{eq:ucl}                                                          
    \parbox{4.2in}{If $\calP$ is a better partial ordering, then the
      upwardly closed subsets of $\calP$ form a better partial
      ordering under reverse inclusion.}
  \end{equation}
  Applying (\ref{eq:dcl}) to the better partial ordering $\NN^n$, we
  see that $\calD_n$ (see Remark \ref{rem:Dn}) is a better partial
  ordering.  And then, applying (\ref{eq:ucl}) to $\calD_n$, we see
  that the upwardly closed subsets of $\calD_n$ are better partially
  ordered under reverse inclusion.  Letting $n=2$, it follows by Lemma
  \ref{lem:KS} that the two-sided ideals of $K[S]$ are better
  partially ordered under reverse inclusion, so in particular they
  satisfy the antichain condition.
\end{proof}

\begin{rem}
  The calculations of Aschenbrenner and Pong
  \cite{monomial-ideals-aschenbrenner-pong} can be used to show that
  Maclagen's result is reverse-mathematically equivalent to
  $\WO(\omega^{\omega^\omega})$.  It is also known \cite[Theorem
  2.9]{hbt} that the Robson Basis Theorem \cite[Theorem 3.15]{robson}
  is reverse-mathematically equivalent to
  $\WO(\omega^{\omega^\omega})$, and we conjecture that $\forall
  n\,\WPO(\calD_n)$ has the same reverse-mathematical status.  In view
  of Theorem \ref{thm:KSacc}, it would be interesting to determine the
  reverse-mathematical status of Theorem \ref{thm:KSwpo}.
\end{rem}

% \newpage % needed if References is at the top of a page
\phantomsection
\addcontentsline{toc}{section}{References}

%\bibliographystyle{plain}
%\bibliography{misc}

\begin{thebibliography}{10}

\bibitem{andrews-tofp}
George~E. Andrews.
\newblock {\em The {T}heory of {P}artitions}.
\newblock Cambridge University Press, 2014.
\newblock XVI + 273 pages.

\bibitem{monomial-ideals-aschenbrenner-pong}
Matthias Aschenbrenner and Wai~Yan Pong.
\newblock Orderings of monomial ideals.
\newblock {\em Fundamenta Mathematicae}, 181(1):27--74, 2004.

\bibitem{iandt}
C.-T. Chong, Q.~Feng, T.~A. Slaman, and W.~H. Woodin, editors.
\newblock {\em Infinity and {T}ruth}.
\newblock Number~25 in IMS Lecture Notes Series, Institute for Mathematical
  Sciences, National University of Singapore. World Scientific, 2014.
\newblock IX + 234 pages.

\bibitem{curtis-reiner-1962}
Charles~W. Curtis and Irving Reiner.
\newblock {\em Representation {T}heory of {F}inite {G}roups and {A}ssociative
  {A}lgebras}.
\newblock Interscience, John Wiley and Sons, 1962.
\newblock XIV + 689 pages.

\bibitem{dejongh-parikh}
D.~H.~J. de~Jongh and Rohit Parikh.
\newblock Well-partial orderings and hierarchies.
\newblock {\em Indagationes Mathematicae}, 80(3):195--207, 1977.

\bibitem{formanek-lawrence}
Edward Formanek and John Lawrence.
\newblock The group algebra of the infinite symmetric group.
\newblock {\em Israel Journal of Mathematics}, 23(3--4):325--331, 1976.

\bibitem{fulton-tableaux-1997}
William Fulton.
\newblock {\em Young {T}ableaux {W}ith {A}pplications to {R}epresentation
  {T}heory and {G}eometry}.
\newblock Number~35 in London Mathematical Society Student Texts. Cambridge
  University Press, 1997.
\newblock IX + 260 pages.

\bibitem{hatz-disguises}
Kostas Hatzikiriakou.
\newblock Algebraic disguises of {$\Sigma^0_1$} induction.
\newblock {\em Archive for Mathematical Logic}, 29(1):47--51, 1989.

\bibitem{hatz-phd}
Kostas Hatzikiriakou.
\newblock {\em Commutative {A}lgebra in {S}ubsystems of {S}econd {O}rder
  {A}rithmetic}.
\newblock PhD thesis, Pennsylvania State University, 1989.
\newblock VII + 43 pages.

\bibitem{hatz-ordinal}
Kostas Hatzikiriakou.
\newblock A note on ordinal numbers and rings of formal power series.
\newblock {\em Archive for Mathematical Logic}, 33(4):261--263, 1994.

\bibitem{hilbert-hbt}
David Hilbert.
\newblock Ueber die {T}heorie der algebraischen {F}ormen.
\newblock {\em Mathematische Annalen}, 36(4):473--534, 1890.
\newblock http://dx.doi.org/10.1007/BF01208503.

\bibitem{hilbert-unendliche}
David Hilbert.
\newblock {\"U}ber das {U}nendliche.
\newblock {\em Mathematische Annalen}, 95(1):161--190, 1926.

\bibitem{kerber-1}
Adalbert Kerber.
\newblock {\em Representations of {P}ermutation {G}roups {I}}.
\newblock Number 240 in Lecture Notes in Mathematics. Springer-Verlag, 1971.
\newblock V + 192 pages.

\bibitem{ko-yo}
Leszek~Aleksander Ko{\l}odziejczyk and Keita Yokoyama.
\newblock Categorical characterizations of the natural numbers require
  primitive recursion.
\newblock {\em Annals of Pure and Applied Logic}, 166(2):219--231, 2015.

\bibitem{monomial-ideals-maclagen}
Diane Maclagen.
\newblock Antichains of monomial ideals are finite.
\newblock {\em Proceedings of the American Mathematical Society},
  129(6):1609--1615, 2000.

\bibitem{mans-weit}
Richard Mansfield and Galen Weitkamp.
\newblock {\em Recursive {A}spects of {D}escriptive {S}et {T}heory}.
\newblock Oxford {L}ogic {G}uides. Oxford University Press, 1985.
\newblock VI + 144 pages.

\bibitem{nw-seq}
Crispin St.\ J.~A. Nash-Williams.
\newblock On better-quasi-ordering transfinite sequences.
\newblock {\em Mathematical Proceedings of the Cambridge Philosophical
  Society}, 64(2):273--290, 1968.

\bibitem{noether-on-gordan}
Max Noether.
\newblock Paul {G}ordan.
\newblock {\em Mathematische Annalen}, 75(1):1--41, 1914.

\bibitem{robson}
J.~C. Robson.
\newblock Well quasi-ordered sets and ideals in free semigroups and algebras.
\newblock {\em Journal of Algebra}, 55(2):521--535, 1978.

\bibitem{schmidt}
Diana Schmidt.
\newblock {\it{Well-{P}artial {O}rderings and {T}heir {M}aximal {O}rder
  {T}ypes}}.
\newblock Habilitationsschrift, University of Heidelberg, 1979, IV + 77 pages.

\bibitem{log-comb}
S.~G. Simpson, editor.
\newblock {\em Logic and {C}ombinatorics}.
\newblock Contemporary {M}athematics. American Mathematical Society, 1987.
\newblock XI + 394 pages.

\bibitem{bqo-fraisse}
Stephen~G. Simpson.
\newblock {BQO} theory and {F}ra{\"\i}ss{\'e'}s conjecture.
\newblock In {\em {\rm\cite{mans-weit}}}, pages 124--138, 1985.

\bibitem{baume}
Stephen~G. Simpson.
\newblock Nichtbeweisbarkeit von gewissen kombinatorischen {E}igenschaften
  endlicher {B\"a}ume.
\newblock {\em Archiv f{\"u}r mathematische Logik und Grundlagenforschung},
  25(1):45--65, 1985.

\bibitem{hbt}
Stephen~G. Simpson.
\newblock Ordinal numbers and the {H}ilbert basis theorem.
\newblock {\em Journal of Symbolic Logic}, 53(3):961--974, 1988.

\bibitem{partial}
Stephen~G. Simpson.
\newblock Partial realizations of {H}ilbert{'}s program.
\newblock {\em Journal of Symbolic Logic}, 53(2):349--363, 1988.

\bibitem{sosoa}
Stephen~G. Simpson.
\newblock {\em Subsystems of {S}econd {O}rder {A}rithmetic}.
\newblock Perspectives in Mathematical Logic. Springer-Verlag, 1999.
\newblock XIV + 445 pages; Second Edition, Perspectives in Logic, Association
  for Symbolic Logic, Cambridge University Press, 2009, XVI + 444 pages.

\bibitem{bct}
Stephen~G. Simpson.
\newblock Baire categoricity and {$\Sigma^0_1$} induction.
\newblock {\em Notre Dame Journal of Formal Logic}, 55(1):75--78, 2014.

\bibitem{actual}
Stephen~G. Simpson.
\newblock An objective justification for actual infinity?
\newblock In {\em {\rm\cite{iandt}}}, pages 225--228, 2014.

\bibitem{toim}
Stephen~G. Simpson.
\newblock Toward objectivity in mathematics.
\newblock In {\em {\rm\cite{iandt}}}, pages 157--169, 2014.

\bibitem{WO-vs-IND-arxiv}
Stephen~G. Simpson.
\newblock Comparing {WO$(\omega^\omega)$} with {$\Sigma^0_2$} induction.
\newblock \href{http://arxiv.org/abs/1508.02655}{arXiv:1508.02655}, 11 August
  2015.
\newblock 6 pages.

\bibitem{simpson-smith}
Stephen~G. Simpson and Rick~L. Smith.
\newblock Factorization of polynomials and {$\Sigma^0_1$} induction.
\newblock {\em Annals of Pure and Applied Logic}, 31:289--306, 1986.

\bibitem{rmpc}
Stephen~G. Simpson and Keita Yokoyama.
\newblock Reverse mathematics and {P}eano categoricity.
\newblock {\em Annals of Pure and Applied Logic}, 164(3):284--293, 2013.

\bibitem{tait}
William~W. Tait.
\newblock Finitism.
\newblock {\em Journal of Philosophy}, 78(9):524--546, 1981.

\bibitem{vems}
Fons van Engelen, Arnold~W. Miller, and John Steel.
\newblock Rigid {B}orel sets and better quasi-order theory.
\newblock In {\em {\rm\cite{log-comb}}}, pages 199--222, 1987.

\bibitem{yokoyama-rt-weak}
Keita Yokoyama.
\newblock On the strength of {R}amsey's {T}heorem without
  {$\Sigma_1$}-induction.
\newblock {\em Mathematical Logic Quarterly}, 59(1--2):108--111, 2013.

\end{thebibliography}

\end{document}